\documentclass[12pt]{article}
\usepackage{geometry}                
\usepackage{graphicx, amsmath, amssymb, amsthm, tikz, float, mathtools, subcaption,algorithm, algorithmic,MnSymbol}
\newtheorem{thm}{Theorem}[section]

\newtheorem{lemm}[thm]{Lemma}
\newtheorem{cor}[thm]{Corollary}

\newtheorem{definition}[thm]{Definition}

\usepackage{lineno,hyperref}

\begin{document}
\begin{center}
\uppercase{\bf Cops and an Insightful Robber}
\vskip 20pt
{\bf Melissa A. Huggan}\\
{\small Department of Mathematics, Ryerson University\\Toronto, Ontario, Canada}\\
{\tt melissa.huggan@ryerson.ca}\\ 
\vskip 10pt
{\bf Richard J. Nowakowski}\\
{\small Department of Mathematics and Statistics, Dalhousie University\\Halifax, Nova Scotia, Canada}\\
{\tt r.nowakowski@dal.ca}\\ 
\end{center}

\begin{abstract}
The `Cheating Robot' version of Cops and Robbers is played on a finite, simple, connected graph.  The players move in the same time period. However, before moving, the robot observes to which vertices the cops are moving and it is fast enough to complete its move in the time period. The cops also know that the robot will use this information. More cops are required to capture a robot than to capture a robber. Indeed, the minimum degree is a lower bound on the number of cops required to capture a robot. Only on a tree is one cop guaranteed to capture a robot, although two cops are sufficient to capture both a robber and a robot on outerplanar graphs. In graphs where retracts are involved, we show how cop strategies against a robber can be modified to capture a robot. This approach gives exact numbers for hypercubes, and $k$-dimensional grids in general.

\vspace{0.5cm}

\end{abstract}

\section{Introduction}  
In Pursuit-Evasion games such as Cops and Robbers, the players move alternately. This is not true in `real life'. Also, the Bad Guys often have access to  information (via bribery or bugging offices and equipment) about the Good Guys strategies. How this changes the Good Guys strategies is the subject of this paper. 

For the \emph{Cheating Robot} (CR) variant of Cops and Robbers, the board is a finite, simple, connected graph.   One player controls $k$ tokens, called \textit{cops}, and the other player controls one token, called the \textit{robot}. In accordance with usual practice, the cops are female and the robot is male.  Initially, the cops place their tokens on vertices, then the robot places their token. A vertex can have more than one token. Thereafter, they move each of their tokens to another vertex in its closed neighbourhood. The players move in the same time period, however, the robot moves immediately after divining the cops' moves. 

Play can be regarded as alternating, and we will adopt this point of view. However, the capturing rules must be changed to reflect the timing of the moves.
A cop \textit{captures} the robot either (i) if, at the end of the robot's move, a cop and the robot are on the same vertex, or (ii) the robot traverses
an edge traversed by a cop on their last move. We assume that the cops know that the robot has the information about their moves when they are determining their strategy.

This model was first introduced within the general context of combinatorial games in~\cite{Huggan,HugganRJN}.  For context, the name, 
Cheating Robot, comes from the Japanese robot that wins Rock-Paper-Scissors 100\% of the time against humans. 
It cheats by having processors fast enough to both identify the human's move and respond appropriately \cite{ItoSYI2016,ShortHVS2010}. Even though the players are supposed to move in the same time period, the robot is moving second every time.

 The original game of Cops and Robbers was first considered in \cite{NowaW,Quilliot} and extended in \cite{AignerF}. The players moved alternately, there is perfect information and the robber is caught if, at any point in the game, a cop and the robber are on the same vertex. The game has been extensively studied with many outstanding questions. See  \cite{BonatoN} for the foundations of this area of research. Cops and Robbers belong to a wider class called pursuit-evasion games.  
 
 For a  graph theory reference, see \cite{West}.  The following definitions are standard. 

\begin{definition}
 For a graph $G$, let $c(G)$ and $c_{cr}(G)$ be the least number of cops required to capture the robber and robot, respectively, on $G$.
\end{definition}

 When referring to results, to help make the distinction clear, results about a `robber' will always be referring to the Cops and Robbers game and results about a `robot' will refer to the Cheating Robot variant. Knowing the behaviour of parameters on subgraphs is often helpful. We show that subgraphs with minimum degree constraints give a lower bound for the number of cops needed to capture a robot, Theorem \ref{thm: k-core}. This also leads to showing that $c_{cr}(G)=1$ if and only if $G$ is a tree, Theorem \ref{thm: trees}.
  It is not known if $c_{cr}(H) \leq c_{cr}(G)$  for $H$ an induced subgraph of $G$ but Theorem \ref{thm:retract} shows that $c_{cr}(H) \leq c_{cr}(G)$ if $H$ is a retract. Retracts are useful. Apart from providing a lower bound, they play an important part when considering products of graphs. Theorem \ref{thm:cartesianproduct} proves that the Cartesian product of two graphs requires no more cops than the total needed for two graphs individually. This immediately gives the exact number for hypercubes and $k$-dimensional grids, Corollary \ref{cor:cartesiangrid}. In Cops and Robbers, 
 the strong product of graphs was relatively easy to analyse. The same is not true in the Cheating  Robot model. Even for paths,
 although the two-dimensional case is relatively straightforward, Theorem \ref{thm: strong grid base case}, in general,  only bounds are known,
 Theorem \ref{thm: k-dim strong grids}.  
  For outerplanar graphs, again, retracts help to show that only 2 cops are required to capture the robot, Theorem \ref{thm:outerplanar}. We close with some open questions.
 
Closely related game variants to the Cheating Robot are: (1) simultaneously moving cops and robbers~\cite{Konstantinidis2016} and (2) surrounding cops and robbers (\cite{BradshawH}, \cite{BurgessCCDFJP}). The former investigates a variant where players are all moving simultaneously. Naturally, this is a probabilistic approach. The latter examines a similar ruleset to ours with the caveat that the robber and cop are allowed to cross an edge and the game can end after a cop move. In comparing the model from \cite{BurgessCCDFJP} and ours, a natural question arises: How closely related are the models? More specifically, do the cop numbers differ by at most $1$ for all classes of graphs? This currently remains an open question. 

\section{Degree Constraints}

For any graph $G$, comparing strategies available to the robot and those available to the robber, we obtain a lower bound on the 
required number of cops. 

\begin{thm} For any graph $G$, $c_{cr}(G)\geq c(G)$.
\end{thm}
\begin{proof} If there are $c(G)-1$ cops, the robot can use the robber's strategy in $G$.
\end{proof}

\begin{cor}\label{cor: robot win}
Let $G$ be a graph and the cop player is controlling $k$ cops. If the robber can win on $G$ against $k$ cops then so can the robot. 
\end{cor}
\begin{proof}
The robot adopts the robber's strategy. 
\end{proof}

The converse of Corollary~\ref{cor: robot win} is false. For example, one cop on $C_3$ will capture the robber but not the robot because the robot will always have an escape move. The minimum degree of a graph provides a lower bound for the number of cops required to capture. This idea is extended in the next result. 

\begin{definition} Let $G$ be a graph, $k$ a positive integer, and $\delta_{G}$ be the minimum degree of $G$. An induced subgraph $H$ of $G$ is a $k$-core if 
$\delta_H\geq k$.
\end{definition}

It is known that to find a $k$-core it is sufficient to iteratively choose a vertex $x$ with degree less than $k$, delete $x$,
and continue in the reduced graph.

\begin{thm}\label{thm: k-core} Let $G$ be a graph with a $k$-core for some $k$. Then $c_{cr}(G)\geq k$.
\end{thm}
\begin{proof} Let $H$ be a $k$-core of $G$. The robot restricts himself to only moving in $H$.
Suppose there are $k-1$ cops and suppose at some move the robot is on vertex $x\in V(H)$. By assumption,
 the degree of $x$ in $H$ is at least $k$. If none of the cops  move to $x$ the robot does not move. If $i$ of the cops move 
 from, say $y_1,y_2,\ldots,y_i$ to $x$ then the robot will be captured if he moves to any  $y_j$. However, there are only $k-i-1$ cops to cover the other, at least $k-i$, neighbours of $x$. Consequently, one neighbour will be unoccupied and the robot moves to it.   \end{proof}

\begin{cor}\label{cor: cycles}
For any cycle $C_n$, $c_{cr}(C_n)=2$. 
\end{cor}
\begin{proof}
A cycle is a 2-core and hence at least two cops are required. Two cops can capture by starting on the same vertex and moving around the cycle in different directions.
\end{proof}
 
 To characterize when one cop will suffice is now straightforward.
 
\begin{thm}\label{thm: trees} Let $G$ be a connected graph. Then $c_{cr}(G)=1$ if and only if $G$ is a tree.
\end{thm}

\begin{proof}
If $G$ is a tree, then the cop places herself somewhere close to the centre of the tree, and the robot places himself anywhere. Since $G$ is a tree, there exists a unique shortest path between the cop and the robot. The cop moves along the first edge of this path. The robot cannot pass the cop and is eventually forced onto a leaf and is caught on the next move.

If $G$ is not a tree then there exists a cycle in $G$. This cycle is a 2-core and, by Theorem \ref{thm: k-core}, $c_{cr}(G) >1$. 
\end{proof}

\section{Retracts}
A typical technique  to capture the robber is to capture his image on a subgraph first, see the papers \cite{AignerF,BI,NowaW} and the book \cite{BonatoN} for others. The following definitions are standard and can be found in \cite{BonatoN}.
\begin{definition}\normalfont{\cite{BonatoN}} Given a graph $G$, an induced subgraph $R$ is a \textit{retract} if there is a \textit{retraction} 
map $f: V(G)\rightarrow V(R)$
where (i) if $x\sim y$ then $f(x)\sim f(y)$ or $f(x)=f(y)$, and (ii) if $x\in V(R)$ then $f(x)=x$. 
\end{definition}

For example, let $P_m=\{a_1,a_2,\ldots, a_m\}$ and $P_n=\{b_1,b_2,\ldots,b_n\}$. In both cases, if  $G=P_m\square P_n$, the Cartesian grid, or $G=P_m\boxtimes P_n$, the strong grid, then for each $i$, $\{a_i\}\times P_n$ is a retract of $G$ with $f((a_r,b_s))=(a_i,b_s)$. Similarly, $P_m\times \{b_j\}$ is also a retract of $G$ with $g((a_r,b_s)) = (a_r,b_j)$. Unless otherwise specified, these will be the retraction maps used when dealing with paths.

\begin{definition}\normalfont{\cite{BonatoN}}
Given a graph $G$, a retract $R$, and a retraction map $f:G\rightarrow R$, then $f(x)$ is the  \textit{shadow} of vertex $x$.
\end{definition}

A retraction map is  \textit{edge preserving} in that edges are mapped to edges or the endpoints are mapped to the same vertex.
This means that the shadow of the robot in a retract traces out a walk on the retract.  Retracts are well known in the Cops and Robbers literature (see \cite{NowaW,BI} for example). Retracts are also important substructures in the Cheating Robot model.  
 
 \begin{thm}\label{thm:retract} If $R$ is a retract of  $G$  then  $c_{cr}(R)\leq c_{cr}(G)$. 
\end{thm}
\begin{proof} Let $f:G\rightarrow R$ be a retraction map. We restrict the robot to playing on $R$ but allow $c_{cr}(G)$ cops to play on $G$. Now, consider the situation with $c_{cr}(G)$ cops on
$R$ where each cop in $R$ is the image of the corresponding cop in $G$.  Note that, because $f$ is edge-preserving, any move by a cop in $G$ is a legal move for the image of the cop. The cops on $G$ have a winning strategy. On the last move of this winning strategy, 
the cops move to occupy all, except for possibly one, neighbouring vertices to the robot and the one cop moves to the vertex containing the robot, 
moving from this unoccupied vertex if it exists. The robot has no escape, in $G$, and so has no escape in $R$. These vertices include all the vertices in $R$ and a cop has moved on to the vertex occupied by the robot.  The robot has been caught by the $c_{cr}(G)$ cops in $R$.
\end{proof}

The follow-up result in \cite{BI} has $c(G)\leq \max\{c(R),c(G\setminus R)+1\}$. However, this is not true in the Cheating Robot
model. In Figure~\ref{fig: retraction example}, $R=C_6=G\setminus R$, $c_{cr}(R) = c_{cr}(G\setminus R)=2$, and $R$ is a retract of $G$.
However, every vertex of $G$ has degree 5, i.e., $G$ is a 5-core, thus $c_{cr}(G)\geq 5>3=\max\{c(R),c(G\setminus R)+1\}$.

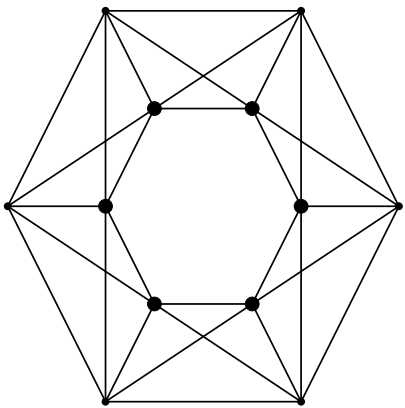
\begin{figure}[htb]
\begin{center}
\scalebox{0.65}{
\begin{tikzpicture}
\draw [line width=1.pt] (-1.,5.)-- (1.,5.);
\draw [line width=1.pt] (-1.,5.)-- (-2.,3.);
\draw [line width=1.pt] (-2.,3.)-- (-1,1.);
\draw [line width=1.pt] (-1.,1.)-- (1,1.);
\draw [line width=1.pt] (2.,3)-- (1,1.);
\draw [line width=1.pt] (2.,3.)-- (1,5.);

\draw [line width=1.pt] (-2.,7.)-- (2.,7.);
\draw [line width=1.pt] (-2.,7.)-- (-4.,3.);
\draw [line width=1.pt] (-4.,3.)-- (-2,-1.);
\draw [line width=1.pt] (-2.,-1.)-- (2,-1.);
\draw [line width=1.pt] (4.,3)-- (2,-1.);
\draw [line width=1.pt] (4.,3.)-- (2,7.);

\draw [ line width=1.pt] (-1.,5.)-- (-2.,7.);
\draw [line width=1.pt] (1.,5.)-- (2.,7.);
\draw [line width=1.pt] (-4.,3.)-- (-2,3.);
\draw [line width=1.pt] (4.,3.)-- (2,3.);
\draw [line width=1.pt] (1.,1)-- (2,-1.);
\draw [line width=1.pt] (-1.,1.)-- (-2,-1.);

\draw [line width=1.pt] (-1.,5.)-- (2.,7.);
\draw [line width=1.pt] (1.,5.)-- (-2.,7.);
\draw [line width=1.pt] (1,5.)-- (4,3.);
\draw [line width=1.pt] (2.,7.)-- (2,3.);
\draw [line width=1.pt] (1.,1)-- (-2,-1.);
\draw [line width=1.pt] (-1.,1.)-- (2,-1.);

\draw [line width=1.pt] (-1,5.)-- (-4,3.);
\draw [line width=1.pt] (-2.,7.)-- (-2,3.);
\draw [line width=1.pt] (1.,1)-- (4,3.);
\draw [line width=1.pt] (2.,3.)-- (2,-1.);
\draw [line width=1.pt] (-1.,1)-- (-4,3.);
\draw [line width=1.pt] (-2.,3.)-- (-2,-1.);
\begin{scriptsize}
\draw [fill=black] (-1.,5.) circle (4.0pt);
\draw [fill=black] (1,5.) circle (4.0pt);
\draw [fill=black] (-1.,1.) circle (4.0pt);
\draw [fill=black] (1,1.) circle (4.0pt);
\draw [fill=black] (-2.,3.) circle (4.0pt);
\draw [fill=black] (2,3.) circle (4.0pt);

\draw [fill=black] (-2.,7.) circle (2.0pt);
\draw [fill=black] (2,7.) circle (2.0pt);
\draw [fill=black] (-2.,-1.) circle (2.0pt);
\draw [fill=black] (2,-1.) circle (2.0pt);
\draw [fill=black] (-4.,3.) circle (2.0pt);
\draw [fill=black] (4,3.) circle (2.0pt);

\end{scriptsize}
\end{tikzpicture}}
\caption{A graph $G$ with retract, $R$, which is the graph induced by the bold vertices.}\label{fig: retraction example}
\end{center}
\end{figure}

Despite this, for this paper, the method by which the cops capture the robot is important when considering  products and we highlight this as a result in its own right.
\begin{lemm}\label{lem:retract} Given a graph $G$, a retract $R$, a retraction map $f:G\rightarrow R$ and $c_{cr}(R)$ cops. If the cops, but not the robot, are restricted to $R$ then the cops can capture the robot's shadow. At the completion of the capture move and every move thereafter, only one cop is needed so as to move onto the shadow as part of the cops' turn. \end{lemm}
\begin{proof}
Since $f$ is edge-preserving then the shadow of the robot, on a turn, can only move to an adjacent vertex. By assumption then,
$c_{cr}(R)$ cops suffice to capture the shadow and at capture one cop, say $c_1$, is on the same vertex as the shadow.  However, the robot is not the shadow and may move so the shadow has moved to an adjacent vertex. The cop $c_1$, as part of the cops turn, can move on to the shadow.
\end{proof}

If each vertex in $R$ is adjacent to at most one vertex in $G\setminus R$ then the cop that follows the robot's shadow blocks the robot's entry from
 $G\setminus R$ back into $R$. In other cases, more cops are needed to follow the shadow in order to prevent the robot entering $R$. The first case is typical of Cartesian products which are considered next.


\section{Cartesian Product of Graphs}
When $G$ is a product then $G$ inherits some properties of the multiplicands. These properties have been useful when determining
 the cop-number of graphs.
In this section we consider the Cartesian product of graphs and the strong product in the next. See \cite{IK} for more on graph products.

\begin{definition} The Cartesian product of graphs $G$ and $H$, written $G\square H$, has $V(G\square H)=\{(g,h): g\in V(G), h\in V(H)\}$
and $(g,h)$ is adjacent to $(g',h')$ if either $g=g'$ and $h$ is adjacent to $h'$ or $g$ is adjacent to $g'$ and $h=h'$.
\end{definition}

This definition is easily extended to the product of more than two graphs. In general, a vertex of the product is $(a_1,a_2,\ldots, a_n)$ where
$a_i$ is in the $i^\text{th}$-coordinate.

Given a graph $F$ with $c_{cr}(F)$ cops, consider the last move by the cops in their winning strategy. All vertices adjacent to the robot
will be occupied except for possibly those from which the cops will move to the robot's vertex. The robot has nowhere safe to move. However, if 
these cops do not move, but the other cops do, then the robot is not captured but neither can he move. We refer to this as the \textit{surrounding} strategy. 
It is an integral part of capturing the robot on the Cartesian product.

For connected graphs $G$ and $H$, To\v{s}i\'{c} \cite{tos} showed that  $c(G\square H)\leq c(G)+c(H)$. The same is true for the Cheating Robot. For an example, see Figure~\ref{fig: cartesian product example}.

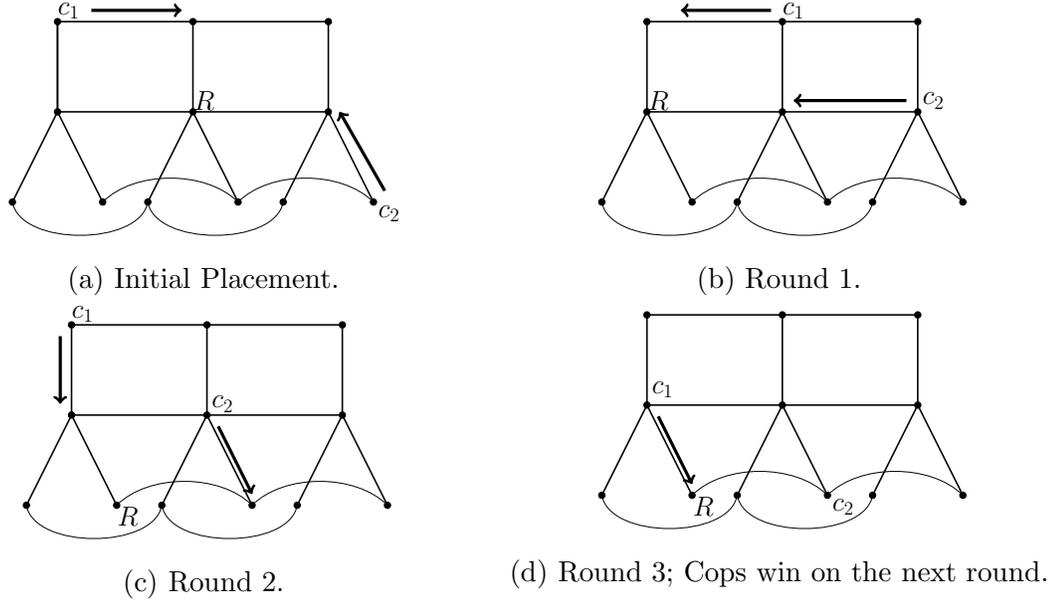
\begin{figure}
\begin{subfigure}{.5\textwidth}
\begin{center}
\scalebox{0.6}{
\begin{tikzpicture}
\draw [->,line width=2.pt] (0.75,2.25)-- (2.75,2.25);
\draw [->,line width=2.pt] (7.25,-1.75)--(6.25,0);
\draw [line width=1.pt] (0.,2.)-- (0.,0.);
\draw [line width=1.pt] (0.,2.)-- (0.,0.);
\draw [line width=1.pt] (3.,2.)-- (3.,0.);
\draw [line width=1.pt] (6.,2.)-- (6.,0.);

\draw [line width=1.pt] (-1.,-2.)-- (0.,0.);
\draw [line width=1.pt] (1.,-2.)-- (0.,0.);

\draw [line width=1.pt] (2.,-2.)-- (3.,0.);
\draw [line width=1.pt] (4.,-2.)-- (3.,0.);

\draw [line width=1.pt] (5.,-2.)-- (6.,0.);
\draw [line width=1.pt] (7.,-2.)-- (6.,0.);

\draw [line width=1.pt] (0.,0.)-- (3.,0.);
\draw [line width=1.pt] (3.,0.)-- (6.,0.);

\draw [line width=1.pt] (0.,2.)-- (3,2.);
\draw [line width=1.pt] (3.,2.)-- (6,2.);

\draw (-1,-2) to[out=280,in=260,loop, distance =1cm] (2,-2);
\draw (2,-2) to[out=280,in=260,loop, distance =1cm] (5,-2);

\draw (1,-2) to[out=45,in=135,loop, distance =1cm] (4,-2);
\draw (4,-2) to[out=45,in=135,loop, distance =1cm] (7,-2);
\begin{scriptsize}
\draw [fill=black] (0.,2.) circle (2.0pt);
\draw [fill=black] (3.,2.) circle (2.0pt);
\draw [fill=black] (6.,2.) circle (2.0pt);

\draw [fill=black] (0.,0.) circle (2.0pt);
\draw [fill=black] (3.,0.) circle (2.0pt);
\draw [fill=black] (6.,0.) circle (2.0pt);

\draw [fill=black] (-1.,-2.) circle (2.0pt);
\draw [fill=black] (1.,-2.) circle (2.0pt);
\draw [fill=black] (2.,-2.) circle (2.0pt);
\draw [fill=black] (4.,-2.) circle (2.0pt);
\draw [fill=black] (5.,-2.) circle (2.0pt);
\draw [fill=black] (7.,-2.) circle (2.0pt);
\draw[color=black] (3.25,0.25) node {\Large $R$};
\draw[color=black] (0.25,2.25) node {\Large $c_{1}$};
\draw[color=black] (7.35,-2.25) node {\Large $c_{2}$};
\end{scriptsize}
\end{tikzpicture}}
\caption{Initial Placement.}
\end{center}
\end{subfigure}
\begin{subfigure}{.5\textwidth}
\begin{center}
\scalebox{0.6}{
\begin{tikzpicture}
\draw [->,line width=2.pt] (2.75,2.25)--  (0.75,2.25);
\draw [->,line width=2.pt]  (5.75,0.25)--(3.25,0.25);
\draw [line width=1.pt] (0.,2.)-- (0.,0.);
\draw [line width=1.pt] (3.,2.)-- (3.,0.);
\draw [line width=1.pt] (6.,2.)-- (6.,0.);

\draw [line width=1.pt] (-1.,-2.)-- (0.,0.);
\draw [line width=1.pt] (1.,-2.)-- (0.,0.);

\draw [line width=1.pt] (2.,-2.)-- (3.,0.);
\draw [line width=1.pt] (4.,-2.)-- (3.,0.);

\draw [line width=1.pt] (5.,-2.)-- (6.,0.);
\draw [line width=1.pt] (7.,-2.)-- (6.,0.);

\draw [line width=1.pt] (0.,0.)-- (3.,0.);
\draw [line width=1.pt] (3.,0.)-- (6.,0.);

\draw [line width=1.pt] (0.,2.)-- (3,2.);
\draw [line width=1.pt] (3.,2.)-- (6,2.);

\draw (-1,-2) to[out=280,in=260,loop, distance =1cm] (2,-2);
\draw (2,-2) to[out=280,in=260,loop, distance =1cm] (5,-2);

\draw (1,-2) to[out=45,in=135,loop, distance =1cm] (4,-2);
\draw (4,-2) to[out=45,in=135,loop, distance =1cm] (7,-2);
\begin{scriptsize}
\draw [fill=black] (0.,2.) circle (2.0pt);
\draw [fill=black] (3.,2.) circle (2.0pt);
\draw [fill=black] (6.,2.) circle (2.0pt);

\draw [fill=black] (0.,0.) circle (2.0pt);
\draw [fill=black] (3.,0.) circle (2.0pt);
\draw [fill=black] (6.,0.) circle (2.0pt);

\draw [fill=black] (-1.,-2.) circle (2.0pt);
\draw [fill=black] (1.,-2.) circle (2.0pt);
\draw [fill=black] (2.,-2.) circle (2.0pt);
\draw [fill=black] (4.,-2.) circle (2.0pt);
\draw [fill=black] (5.,-2.) circle (2.0pt);
\draw [fill=black] (7.,-2.) circle (2.0pt);
\draw[color=black] (0.25,0.25) node {\Large $R$};
\draw[color=black] (3.25,2.25) node {\Large $c_{1}$};
\draw[color=black] (6.35,0.25) node {\Large $c_{2}$};
\end{scriptsize}
\end{tikzpicture}}
\caption{Round 1.}
\end{center}
\end{subfigure}
\begin{subfigure}{0.5\textwidth}
\begin{center}
\scalebox{0.6}{
\begin{tikzpicture}
\draw [->,line width=2.pt] (-0.25,1.75)--  (-0.25,0.25);
\draw [->,line width=2.pt]  (3.25,-0.25)--(4.0,-1.75);
\draw [line width=1.pt] (0.,2.)-- (0.,0.);
\draw [line width=1.pt] (3.,2.)-- (3.,0.);
\draw [line width=1.pt] (6.,2.)-- (6.,0.);

\draw [line width=1.pt] (-1.,-2.)-- (0.,0.);
\draw [line width=1.pt] (1.,-2.)-- (0.,0.);

\draw [line width=1.pt] (2.,-2.)-- (3.,0.);
\draw [line width=1.pt] (4.,-2.)-- (3.,0.);

\draw [line width=1.pt] (5.,-2.)-- (6.,0.);
\draw [line width=1.pt] (7.,-2.)-- (6.,0.);

\draw [line width=1.pt] (0.,0.)-- (3.,0.);
\draw [line width=1.pt] (3.,0.)-- (6.,0.);

\draw [line width=1.pt] (0.,2.)-- (3,2.);
\draw [line width=1.pt] (3.,2.)-- (6,2.);

\draw (-1,-2) to[out=280,in=260,loop, distance =1cm] (2,-2);
\draw (2,-2) to[out=280,in=260,loop, distance =1cm] (5,-2);

\draw (1,-2) to[out=45,in=135,loop, distance =1cm] (4,-2);
\draw (4,-2) to[out=45,in=135,loop, distance =1cm] (7,-2);
\begin{scriptsize}
\draw [fill=black] (0.,2.) circle (2.0pt);
\draw [fill=black] (3.,2.) circle (2.0pt);
\draw [fill=black] (6.,2.) circle (2.0pt);

\draw [fill=black] (0.,0.) circle (2.0pt);
\draw [fill=black] (3.,0.) circle (2.0pt);
\draw [fill=black] (6.,0.) circle (2.0pt);

\draw [fill=black] (-1.,-2.) circle (2.0pt);
\draw [fill=black] (1.,-2.) circle (2.0pt);
\draw [fill=black] (2.,-2.) circle (2.0pt);
\draw [fill=black] (4.,-2.) circle (2.0pt);
\draw [fill=black] (5.,-2.) circle (2.0pt);
\draw [fill=black] (7.,-2.) circle (2.0pt);
\draw[color=black] (1.25,-2.25) node {\Large $R$};
\draw[color=black] (0.25,2.25) node {\Large $c_{1}$};
\draw[color=black] (3.35,0.25) node {\Large $c_{2}$};
\end{scriptsize}
\end{tikzpicture}}
\caption{Round 2.}
\end{center}
\end{subfigure}
\begin{subfigure}{0.5\textwidth}
\begin{center}
\scalebox{0.6}{
\begin{tikzpicture}
\draw [->,line width=2.pt] (0.25,-0.25)--  (1.,-1.75);
\draw [line width=1.pt] (0.,2.)-- (0.,0.);
\draw [line width=1.pt] (3.,2.)-- (3.,0.);
\draw [line width=1.pt] (6.,2.)-- (6.,0.);

\draw [line width=1.pt] (-1.,-2.)-- (0.,0.);
\draw [line width=1.pt] (1.,-2.)-- (0.,0.);

\draw [line width=1.pt] (2.,-2.)-- (3.,0.);
\draw [line width=1.pt] (4.,-2.)-- (3.,0.);

\draw [line width=1.pt] (5.,-2.)-- (6.,0.);
\draw [line width=1.pt] (7.,-2.)-- (6.,0.);

\draw [line width=1.pt] (0.,0.)-- (3.,0.);
\draw [line width=1.pt] (3.,0.)-- (6.,0.);

\draw [line width=1.pt] (0.,2.)-- (3,2.);
\draw [line width=1.pt] (3.,2.)-- (6,2.);

\draw (-1,-2) to[out=280,in=260,loop, distance =1cm] (2,-2);
\draw (2,-2) to[out=280,in=260,loop, distance =1cm] (5,-2);

\draw (1,-2) to[out=45,in=135,loop, distance =1cm] (4,-2);
\draw (4,-2) to[out=45,in=135,loop, distance =1cm] (7,-2);
\begin{scriptsize}
\draw [fill=black] (0.,2.) circle (2.0pt);
\draw [fill=black] (3.,2.) circle (2.0pt);
\draw [fill=black] (6.,2.) circle (2.0pt);

\draw [fill=black] (0.,0.) circle (2.0pt);
\draw [fill=black] (3.,0.) circle (2.0pt);
\draw [fill=black] (6.,0.) circle (2.0pt);

\draw [fill=black] (-1.,-2.) circle (2.0pt);
\draw [fill=black] (1.,-2.) circle (2.0pt);
\draw [fill=black] (2.,-2.) circle (2.0pt);
\draw [fill=black] (4.,-2.) circle (2.0pt);
\draw [fill=black] (5.,-2.) circle (2.0pt);
\draw [fill=black] (7.,-2.) circle (2.0pt);
\draw[color=black] (1.25,-2.25) node {\Large $R$};
\draw[color=black] (0.35,0.35) node {\Large $c_{1}$};
\draw[color=black] (4.35,-2.25) node {\Large $c_{2}$};
\end{scriptsize}
\end{tikzpicture}}
\caption{Round 3; Cops win on the next round.}
\end{center}
\end{subfigure}
\caption{Example of game play on $P_{3}\square S_{3}$, using the strategy from Theorem~\ref{thm:cartesianproduct}. Arrows indicate where the cops ($c_{1}$, $c_{2}$) will move on the next round; the robot ($R$) is aware of this information before he is required to move.}\label{fig: cartesian product example}
\end{figure}

\begin{thm}\label{thm:cartesianproduct} If G and H are each connected then $c_{cr}(G\square H)\leq c_{cr}(G)+c_{cr}(H)$.
\end{thm}

\begin{proof} 
Let  $r=c_{cr}(G)$, $s=c_{cr}(H)$ and let $G'$ and $H'$ be subgraphs of $G$ and $H$ respectively with $|V(G')|=r$ and $|V(H')|=s$. In $G\square H$, place cops on each vertex of  $G'\square\{h\}$ and $\{g\}\square H'$ for some $g\in V(G)$ and $h\in V(H)$.  Call these the $G$-cops and the $H$-cops respectively. 
 We will show that this number of cops suffices to capture the robot. There are two parts to the strategy. 
First, the $G$-cops ($H$-cops) move to 
eventually have their second (first) coordinate the same as that of the robot. 
Next, they follow the surrounding strategy until both the $G$- and $H$-cops have the robot surrounded.

\textit{Phase 1:} Assume that the robot is at vertex $(a,x)$. If $h\ne x$ then the $G$-cops move to $G'\square\{h'\}$ where $h'$ is closer to $x$ than $h$. If $g\ne a$ then the $H$-cops move to $\{g'\}\square H'$ where $g'$ is closer to $a$ than $g$. In response, the robot  moves to 
 $(a',x')$ where at most one of the coordinates has changed, say to $(a',x)$. The $G$-cops have moved closer in the second coordinate 
 and the $H$-cops have maintained the distance in the first.  Similarly, the $H$-cops will have moved closer if the robot moves to $(a,x')$.
 
\textit{Phase 2:} Thus, on every move one group of cops gets closer (or both if the robot remains on the same vertex)
  and eventually, without loss of generality, the $G$-cops and the robot will be on $G\square \{x\}$. The $G$-cops play their surrounding strategy
 and the $H$-cops move as in Phase 1. If the robot moves to change the first coordinate, we are still in Phase 2 and the $H$-cops have moved closer in the second coordinate. If the robot moves to change the second coordinate we are back in Phase 1 but the $G$-cops have taken one move toward surrounding the
 robot and the $H$-cops are closer in the first coordinate.
 
 Eventually, without loss of generality, we can assume that the $G$-cops have surrounded the robot on  $G\square \{x\}$, say at $(a,x)$. We also assume that the $H$-cops are not on $\{a\}\square H$. 
 
\textit{Phase 3:}  The $G$-cops do not move. The $H$-cops move as in Phase 1. The robot can only not move or change the second coordinate thus $H$-cops have moved closer in the first coordinate. 

On the next move, the $G$-cops surround the robber on some $G\square \{y\}$ and the $H$-cops are closer in the first coordinate. Eventually, the robot will be on some $(a',x')$ and the $G$-cops  will have surrounded the robot on $G\square \{x'\}$ and the $H$-cops are on $\{a'\}\square H$. Again, the robot can only not move or change the second coordinate, thus we are playing in $\{a'\}\square H$. Eventually, the robot will be at some $(a',z)$  surrounded by the $G$-cops on $G\square \{z\}$ and surrounded by the $H$-cops on $\{a'\}\square H$. The robot has nowhere safe to move and will be captured on the next turn.
\end{proof}

This result give exact bounds for some Cartesian products.
\begin{cor}\label{cor:cartesiangrid}
Let $G$ be a graph, $\{T_{n_1}, T_{n_2},\ldots,T_{n_k}\}$ and $\{C_{m_1}, C_{m_2},\ldots,C_{m_\ell}\}$ be sets of trees and cycles respectively.
\begin{enumerate}
\item If  $G=\square_{i=1}^kT_{n_i}$ then  $c_{cr}(G) =k$.
\item If  $G=\square_{j=1}^\ell C_{m_j}$ then  $c_{cr}(G) =2\ell$.
\end{enumerate}
\end{cor}

\begin{proof} First note that $G=\square_{i=1}^kT_{n_i}$ has minimum degree $k$ so at least $k$ cops are required. Now, on a
tree one cop suffices to capture the robot therefore, by Theorem \ref{thm:cartesianproduct}, $k$ cops suffice on $\square_{i=1}^kT_{n_i}$.

Similarly, at least $2\ell$ cops are required since this is the minimum degree. Also, $c_{cr}(C_{m_{i}})=2$  for any cycle $C$, thus $2\ell$ cops are suffice 
by Theorem \ref{thm:cartesianproduct}.
\end{proof}

Note that the first result of Corollary \ref{cor:cartesiangrid} covers hypercubes when the trees are each $P_{2}$ and  $k$-dimensional grids when the trees are paths in general. 

\section{Strong Products of Paths}
\begin{definition} The strong  product of graphs $G$ and $H$, written $G\boxtimes H$, has $V(G\boxtimes H)=\{(g,h): g\in V(G), h\in V(H)\}$
and $(g,h)$ is adjacent to $(g',h')$ if  $g$ is equal or adjacent to $g'$ and $h$ is equal or adjacent to $h'$.
\end{definition}

As with the Cartesian product, this definition is easily extended to the product of more than two graphs. In general, a vertex of the product is $(a_1,a_2,\ldots, a_n)$ where
$a_i$ is in the $i^\text{th}$-coordinate.
If $G$ is the strong product of paths, a \emph{corner} is a vertex where for each $i$, the $i^\text{th}$-coordinate is a leaf in $P_{i}$.
For example, there are $4$ corners in a grid, $P_m\boxtimes P_n$, where $P_{i}$ is a path with $i$ vertices. 

\begin{thm}\label{thm: strong grid base case}
Let $G$ be a strong grid, $G = P_{m} \boxtimes P_{n}$. Then
\[c_{cr}(G) =  \begin{cases}
                      4, \text{ if $m,n>3$;}\\ 
                      3, \text{ if $\min\{m,n\} = 2$ or $3$.}\\
				
\end{cases}
\]  
\end{thm}
\begin{proof}
Suppose  $G=P_2\boxtimes P_n$. The minimum degree of $G$ is 3, so at least 3 cops are required. The following algorithm shows that 3 cops suffice.

Place two cops on $(1,1)$ and one on $(2,1)$, the robot is on some vertex $(i,j)$ with $j>1$. Move the first two cops to $(1,2)$ and $(2,2)$, the third stays on $(2,1)$. The robot cannot move to $(1,1)$. The cop on $(2,1)$ now moves to $(2,2)$ and after the robot moves he is on some vertex $(i',j')$ with $j'>2$. The cops repeat until the robot is caught  on $(1,n)$ or $(2,n)$.

Suppose  $G=P_3\boxtimes P_n$, $n\geq 3$. The minimum degree of $G$ is 3, so at least 3 cops are required. The following algorithm shows that 3 cops suffice.

 In general, we assume that cops are on $(1,i)$, $(2,i)$ and $(3,i)$, starting with $i=1$. The robot is on some $(k,j)$ with $j>1$.
 
 On the cops move, if the robot is on $(k,j)$, $j>i+1$ then the cops increase their second coordinate by 1. 
 
  If the robot is on $(1,i+1)$ then the cop on $(3,i)$ moves to $(2,i+1)$. There are two cases. One, if the robot passes then the 
 cop on $(2,i)$ moves to $(1,i+1)$ and now he has to move or be captured. He has to move to $(k,i+2)$, for $k=1$ or $2$, and the other two cops now move to 
 $(2,i+1)$ and $(3,i+1)$.  Two, if the robot moves then
 he must move to $(k,i+2)$, for some $k$ or else be captured.  The other two cops move to $(1,i+1)$ and $(3,i+1)$. In both cases, the cops have increased their second coordinate and the robot is on smaller subgraph.

   If the robot is on $(2,i+1)$ then the cop on $(3,i)$ moves to $(2,i+1)$. There are three cases. One, if the robot moves to $(k,i+2)$ then cops move to occupy $(1,i+1)$, $(2,i+1)$ and $(3,i+1)$. Two, the robot moves  to $(1,i+1)$ and this situation is covered in the previous paragraph.  Three, the robot moves to $(3,i+1)$ whence the other two cops move to $(2,i)$ and $(3,i)$ giving the symmetric case to case 2.  In all cases, the end result is that the cops have increased their second coordinate and the robot is on smaller subgraph. 
   
   Eventually, the robot is captured  in the next sequence after the cops occupy $(1,n-1)$, $(2,n-1)$ and $(3,n-1)$.

Now suppose that $m,n>3$. There exists an induced $4$-core in $G$ (from deleting the four corners), hence four cops are necessary. 

We must show that $4$ cops is also sufficient. To do so, we present an algorithm for the cops to capture the robot. Again, we will take the alternate move interpretation.  In this algorithm, `capture of the shadow' by the cops means that after the cops' move, the robot and some cop are in the same column. The robot still has a move to try and evade the cops.

Let the strong grid be of size $m \times n$. The vertices, indexed by the rows and columns, are $(i,j)$, $1\leq i\leq m$, $1\leq j \leq n$. Increasing the first coordinate corresponds to \textit{moving up} a row. Similarly, increasing the second coordinate corresponds to \textit{moving right}. 
A \textit{4-line} of cops, $(i, \langle j\rangle)$,  has the cops on $(i, j-1)$, $(i, j)$, $(i, j+1)$, and $(i,j+2)$. With this notation, if  $s>n$ then $s$ is taken as $n$
and if $s<1$ then $s$ is taken as 1.
We refer to $(i, j)$ and $(i, j+1)$ as the \textit{interior} vertices of the $4$-line.

Placement: Place the cops on the vertices of $(1,\langle 2\rangle)$.  The robot places himself  on $(i,j)$.

\begin{enumerate}
\item[Step 1:] The cops move the $4$-line along the first row as a group until the shadow of the robot is either (i) captured on an interior vertex of the $4$-line or (ii) the robot is either on $(i, n)$ or $(i, 1)$ and the $4$-line is at $(1, \langle n-2\rangle)$ or $(1,\langle 2\rangle)$, respectively. After this shadow  capture, move to Step 2. 
\item[Step 2:] The robot moves but, regardless of the move,  its shadow is still caught. If the robot is at least two rows above the cops, move the $4$-line up one row to keep the shadow on an interior vertex. 
 If the $4$-line is on a row end, then he moves up. 
Repeat Step 2 until the cops are one row away from the robot. 
\item[Step 3:] It is now the robot's move, where the $4$-line is at $(i,\langle j\rangle)$ and the robot is at $(i+1, j)$ or $(i+1, j+1)$. 
There are a two cases to consider:  (i) If the robot moves up one row then the $4$-line follows keeping the shadow on an interior vertex.  (ii) If the robot does not move up then he moves to (without loss of generality) to either $(i+1,j)$ or $(i+1,j-1)$. In the former,  only the cop on $(i, j+2)$ moves and moves to $(i+1, j+1)$. In the latter, the cop on $(i, j+2)$ moves  $(i+1, j+1)$ and the others move one column to the left. In both cases the situation is now, the robot is on $(i+1,k)$ and
the cops are on $(i+1,k+1)$ and $(i,k-1)$, $(i,k)$, and $(i,k+1)$. If the robot does not move up a row he can only pass or move to $(i,k-1)$. If he 
passes, the cop on $(i+1,k+1)$ moves to $(i+1,k)$, the others pass. To avoid capture, now the robot must move up or left on the same row. If the robot moves up, in all cases the cops can move the four cops to form a $4$-line on the $i+1$ row with the robot's shadow captured by an interior vertex. If the robot moves left then all the cops also move left and the situation is repeated. Eventually, either the robot is on $(i,1)$ and must move up or is on $(m,1)$ and is captured when the cop on $(m,2)$ moves to $(m,1)$.
\end{enumerate}
Thus four cops is sufficient for capture on a strong grid.
\end{proof}

Next, we generalize Theorem~\ref{thm: strong grid base case} to higher dimensions. In a $k$-dimensional strong grid, $G=\boxtimes_{i=1}^kP_{n_i}$,  for brevity, a vertex can represented by $(i_1, i_2, \dots, i_k)$, where $1\leq i_j\leq n_j$. Implicitly, we are setting  $V(P_{n_i})=\{1,2,\ldots,n_i\}$. The paths are  disjoint and the context will indicate to which path a vertex belongs.

\begin{thm}\label{thm: k-dim strong grids}
Let $G$ be a $k$-dimensional strong grid, $G = \boxtimes_{i=1}^{k} P_{n_{i}}$. Then
\begin{equation*}
3\cdot2^{k-1} -2 \leq c_{cr}(\boxtimes_{i=1}^{k} P_{n_{i}}) \leq 3^{k}.
\end{equation*}
\end{thm}
\begin{proof}
The vertices of minimum degree in $G$ are the `corners', i.e. the vertices where all the coordinates are all leaves in the original paths. Their degrees are $2^k-1$. Let $L$ be the vertices of $G$ where all  $k$ coordinates are leaves in the original paths and let $H$ be the induced graph on $V(G)\setminus L$.
 The vertices of minimum degree in $H$ are those with $k-1$ coordinates being leaves in the original paths and the other coordinate being adjacent to leaves. Their degrees in $G$ are $3\cdot2^{k-1} - 1$ and, in addition, the adjacent corner has been removed leaving the degree as $3\cdot2^{k-1} -2$ in $H$.
 
 To capture the robot we give an algorithm based on the cop-win strategy. 
  To avoid special cases when the cops are close to the edges or corners of this strong hypergrid, let 
 $H=\boxtimes_{i=1}^kP'_{n_i}$ where $V(P'_{n_i})= \{0,1,2,\ldots,n_i, n_{i}+1\}$, that is $P_{n_i}$ extended by a new leaf at each end. Note that $P_{n_i}$ is a retract of $P'_{n_i}$ via $f_i(j) = j$ if $1\leq j\leq n_i$, $f_i(0)=1$ and $f_i(n_i+1)=n_i$. In turn, $G$ is a retract of $H$ via $f(v_1,v_2,\ldots,v_k) = (f_1(v_1),f_2(v_2),\ldots,f_k(v_k))$. We now use Theorem \ref{thm:retract} and restrict the robot to $G$ but let the cops play on $H$.

Again, we consider the alternating model with the appropriate restrictions and winning condition. Place the cops on the vertices $(x_1,x_2,\ldots,x_k)$, $1\leq x_i\leq 3$ for each $i$. The cop on $(2,2,\ldots,2)$ is the \textit{central} cop and the others occupy every vertex of its neighbourhood.
The central cop plays the cop-win strategy on $H$. When the central cop moves from $(y_1,y_2,\ldots,y_k)$ to $(y'_1,y'_2,\ldots,y'_k)$  each coordinate changes by at most 1. The non-central cops now move to change their coordinates by 
$(y'_1-y_1,y'_2-y_2,\ldots,y'_k-y_k)$ and they still occupy every vertex of the central cop's neighbourhood. Thus, when the central cop is on the same vertex as the robot, the robot has nowhere safe to move.
\end{proof}

In the proof of the upper bound in Theorem \ref{thm: k-dim strong grids}, note that the robot will have been captured before the central cop captures via the cop-win strategy. This indicates that the upper bound is probably not tight. Indeed, in the 2-dimensional case, the lower bound is correct.
\section{Outerplanar}

An \textit{outerplanar}  graph can be drawn so that the vertices can be arranged in a circle and all edges can be placed on or outside of the circle with no crossing edges.  
We will call edges not lying directly on the circle, \textit{chords}. Let $V(G)=\{v_0,v_1,\ldots,v_n\}$ where the vertices are placed in that order around the circle. The \textit{length} of a chord $v_iv_j$, $i<j$,  is $\min\{|j-i|,|n+1+i-j|\}$, i.e., the shortest arc along the circle. 

Clarke \cite{Clarke} showed that two cops were sufficient to capture a robber. We show that two are sufficient to capture a robot.

\begin{thm}\label{thm:outerplanar}
Let $G$ be a connected  outerplanar graph. If $G$ is a tree then   $c_{cr}(G) = 1$, otherwise $c_{cr}(G) = 2$. 
\end{thm}

\begin{proof} Theorem \ref{thm: trees} covers the case when $G$ is a tree. Thus, we may assume that $G$ is not a tree and therefore,
again by Theorem \ref{thm: trees}, we know $c_{cr}(G)\geq 2$. It remains to show that two cops suffice.

We proceed by induction on $|V(G)|$. If $G$ has one vertex then one cop suffices. We now assume that for some $n\geq 1$, 2 cops suffice to 
capture the robot on any connected outerplanar graph with $k\leq n$ vertices. Let $G$ be a connected outerplanar graph on $n+1$ vertices. Suppose $G$ contains a cut-vertex $x$ and $G\setminus\{x\}$ has a connected component $X'$ such that $X'\cup\{x\}$ is a tree. Let $f:G\rightarrow G\setminus X'$ by $f(y) = y$ if $y\in G\setminus X$ and $f(y)=x$ otherwise. Now $f$ is a retraction map and
thus  $G\setminus X$ is a retract of $G$. Moreover, $G\setminus X$ is connected and outerplanar. By induction, 2 cops suffice to capture the shadow on $G\setminus X$. If the shadow is not on $x$ then the robot has been captured. If the shadow is on $x$ the robot is not captured then he is on $X'$. Now one cop remains on $x$, which prevents the robot moving off $X'$, and by Theorem \ref{thm: trees}, the other cop can capture the robot.

We may suppose that all vertices are of degree 2 or greater.  If $G$ is a cycle then two cops suffice. If $G$ is not a cycle then
there is a chord and let $v_iv_j$, $i<j$,  be a chord of shortest length. Note $j\geq i+2$. Without loss of generality, we may assume that there is no chord $v_rv_s$ with $i\leq r<s\leq j$ unless $i=r$ and $s=j$. Let $R=G\setminus \{v_{i+1},v_{i+2},\ldots, v_{j-1}\}$. Set $f:G\rightarrow R$ by
$f(x) = x$, if $x\in R$, and $f(x) = v_i$ otherwise. Now $f$ is a retraction map and $R$ is a retract. The graph $R$ is a connected, outerplanar graph and by induction, 2 cops can capture the robot's shadow on $R$. 
Once captured, either the robot is captured or 
the shadow is on $v_i$ and the robot is on one of $v_{i+1},v_{i+2},\ldots, v_{j-1}$. The game is over in the former situation so we suppose the latter. Note that if the robot had been on $v_i$ then he would have moved to $v_{i+1}$. Since the shadow has been captured in $R$ on $v_i$ then $v_i$ only has two neighbours in $R$, those being $v_{i-1}$ and $v_j$.  Hence on the cops move which captured the shadow, one cop was on either $v_{j+1}$ or $v_j$ and the other on $v_{i-1}$ or $v_i$. One capturing move is to move the cop on $v_{i-1}$ to $v_i$, or remain on $v_i$,  and the other to move to or remain on $v_j$. The robot is now trapped on the path $v_{i+1},v_{i+2},\ldots, v_{j-1}$. The cops can capture by having the one on $v_i$ move successively along this path. \end{proof}

\section{Open Problems} 

Given the similarities between the surrounding model \cite{BurgessCCDFJP} and the Cheating Robot, we observe that for any graph $G$ we have $c_{cr}(G) \leq s(G)$, where $s(G)$ is the surrounding cop number. For paths with $4$ or more vertices, outerplanar graphs and complete graphs, we know that $0\leq s(G)-c_{cr}(G) \leq 1$.  \emph{Is it true that $s(G)$ and $c_{cr}(G)$ differ by at most 1?}

An important Cops and Robbers result is that three cops suffice to capture a robber on a planar graph \cite{AignerF}. For the surrounding model the upper bound is 7 \cite{BradshawH}. The icosahedron is regular of degree 5, hence there is at least one planar graph that requires at least $5$ cops for capture. We ask: \textit{Is it true that for all planar graphs $G$, $c_{cr}(G)\leq 5$?}

Let $H$ be a connected subgraph of $G$. For Cops and Robbers, there are examples where $c(H)>c(G)$. For the Cheating Robot, \textit{is it always true that 
$c_{cr}(H)\leq c_{cr}(G)$?} Theorem \ref{thm:retract} shows the inequality is true when $H$ is a retract.

\section*{Acknowledgements}

\noindent
Melissa A. Huggan was supported by the Natural Sciences and Engineering Research Council of Canada (funding reference number PDF-$532564-2019$). 
Richard J. Nowakowski was supported by the Natural Sciences and Engineering Research Council of Canada (funding reference number 2019-04914).


\end{document}